\documentclass[a4paper,10pt]{article}

\usepackage[backref]{hyperref}

\usepackage{color}
\usepackage[utf8]{inputenc}
\usepackage{amsmath}
\usepackage{amsthm}
\usepackage{amsrefs}
\usepackage{amssymb}
\usepackage{mathrsfs}
\usepackage{amscd}
\usepackage{mathtools}
\usepackage{tikz-cd}
\usepackage{verbatim}
\usepackage{slashed}
\usepackage[all]{xy}

\usepackage{soul}
\newcounter{commentcounter}

\usepackage{color}

\usepackage{ifthen,srcltx}
\newcommand{\showcomments}{yes}

\newsavebox{\commentbox}
\newenvironment{claima}[1]{\par\noindent \underline{Claim A:}\space#1}{}
\newenvironment{claimb}[1]{\par\noindent \underline{Claim B:}\space#1}{}

\newenvironment{com}%
{\ifthenelse{\equal{\showcomments}{yes}}%
{\footnotemark
        \begin{lrbox}{\commentbox}
        \begin{minipage}[t]{1.25in}\raggedright\sffamily\tiny
        \footnotemark[\arabic{footnote}]}
{\begin{lrbox}{\commentbox}}}
{\ifthenelse{\equal{\showcomments}{yes}}
{\end{minipage}\end{lrbox}\marginpar{\usebox{\commentbox}}}
{\end{lrbox}}}

\title{Relative Torsion and Bordism Classes of Positive Scalar Metrics on Manifolds with Boundary}
\author{Simone Cecchini, Mehran Seyedhosseini, Vito Felice Zenobi}
\newtheorem{theorem}{Theorem}[section]
\newtheorem{proposition}[theorem]{Proposition}
\newtheorem{lemma}[theorem]{Lemma}
\theoremstyle{definition}
\newtheorem{definition}[theorem]{Definition}
\theoremstyle{definition}
\newtheorem{remark}[theorem]{Remark}
\newtheorem{corollary}[theorem]{Corollary}

\DeclareMathOperator{\ind}{ind}

\newcommand{\reals}{\mathbb{R}}

\date{}
\begin{document}
\maketitle
\begin{abstract}
In this paper, we define a relative $L^2$-$\rho$-invariant for Dirac operators on odd-dimensional spin manifolds with boundary and show that they are invariants of the bordism classes of positive scalar curvature metrics which are collared near the boundary. As an application, we show that if a $4k+3$-dimensional spin manifold with boundary admits such a metric and if, roughly speaking, there exists a torsion element in the difference of the fundamental groups of the manifold and its boundary, then there are infinitely many bordism classes of such psc metrics on the given manifold. This result in turn implies that the moduli-space of psc metrics on such manifolds has infinitely many path components. We also indicate how to define delocalised $\eta$-invariants for odd-dimensional spin manifolds with boundary, which could then be used to obtain similar results for $4k+1$-dimensional manifolds.
\end{abstract}
\section{Introduction}
In recent years, there have been many contributions to the study of the existence and classification problem for positive scalar curvature metrics on manifolds. The index theory of (twisted) Dirac operators on spin manifolds has been extensively and successfully used in this study. Some results on the topology of the space of psc metrics can be found in \cite{LM89}, \cite{BG95}, \cite{PS07II}, \cite{XYZ17}, \cite{BERW17} and \cite{ERW19}. Apart from \cite{BERW17} and \cite{ERW19}, most of these results apply to closed manifolds and statements on the existence and classification problem for psc metrics on manifolds with boundary are more scarce in the literature. The recent works \cite{BERW17} and \cite{ERW19} give very precise information about the (nontriviality) of the homotopy groups of the space of psc metrics which are collared near the boundary and restrict to a fixed metric near the boundary. However, the invariants utilised by them cannot be used to gain information about the space of bordism classes of such metrics. The following paper is a contribution to the study of psc metrics on manifolds with boundary. Here, we develop invariants which allow us to study the bordism classes of psc metrics on a manifold with boundary which are collared near the boundary. Our observations lead to results concerning the path components of the (moduli) space of psc metrics which do not necessarily restrict to a fixed metric at the boundary.

In the following, we will not discuss the question whether a given spin manifold with boundary admits a metric with positive scalar curvature (with certain boundary conditions) and we will concern ourselves with the secondary question of classification of such metrics. More precisely, we will show that under certain conditions on the fundamental groups of the manifold and its boundary and the dimension of the manifold, the space of psc metrics which are collared near the boundary has infinitely many path components. The same statement is true for the moduli space of such metrics (modulo the action of the diffeomorphism group). Furthermore, we show that on manifolds satisfying our conditions, there exist infinitely many psc metrics (collared near the boundary) which are not even bordant (in an equivariant sense) to each other. 

In \cite{BG95}, Botvinnik and Gilkey showed that if a closed odd dimensional spin manifold with finite fundamental group admits a positive scalar curvature metric, then it admits infinitely many such metrics which are not (equivariantly) bordant to each other. This implies, in particular, that these metrics do not lie in the same path component of the space of psc metrics. The invariants they used to differentiate psc metrics was the APS-$\rho$-invariants introduced in \cite{APSII75}. These are the $\eta$-invariants of the spin Dirac operator twisted with virtual representations of virtual dimension $0$. In \cite{PS07II}, Piazza and Schick used the approach of Botvinnik and Gilkey and obtained similar results for manifold with arbitrary fundamental groups containing torsion. The invariant used by them was the $L^2$-$\rho$-invariant of Cheeger and Gromov introduced in \cite{CG85}. This is the difference of the the $L^2$-$\eta$-invariant of the Dirac operator associated the universal cover and the $\eta$-invariant of the manifold. It can also be seen as the generalisation of the APS-$\rho$-invariant associated to a virtual representation arising from the regular and trivial representations of the fundamental group. To obtain results in dimesnions $4k+1$, Piazza and Schick also made use of the delocalised $\eta$-inavriant of John Lott. 

The main novelty of this paper is the definition of the relative $L^2$-$\rho$-invariant for odd-dimensional compact spin manifolds with boundary. In most of the paper, we are concerned with establishing the fundamental properties of this invariant. The geometric applications stated above are then obtained by following the same path as the one taken in \cite{PS07II}. In a remark, we also point out how our technique can be used to define delocalised $\eta$-invariants associated to metrics on manifolds with boundaries. Using this, there does not seem to be any problem in obtaining the analogues of the more quantitative results in \cite{PS07II} for manifolds with boundary.

We quickly describe how our relative $L^2$-$\rho$-invariant is defined. Let $M$ be a compact spin manifold with boundary $N$. Suppose that the Riemannian metric is collared near the boundary. By the van Kampen theorem the double of $M$ along $N$ has fundamental group $\pi_1(M)*_{\pi_1(N)}\pi_1(M)$. Denote by $\mathrm{Im}(\pi_1(N))^{\pi_1(M)}$ the normal closure of the image of $\pi_1(N)$ in $\pi_1(M)$ under the map of fundamental groups induced by the inclusion $N \rightarrow M$. One can then twist the spin Dirac operator of the double with a bundle of Hilbert modules over the group von Neumann algebra of $\pi_1(M)^{}/\mathrm{Im}(\pi_1(N))^{\pi_1(M)}$. This bundle is defined using a suitable representation of $\pi_1(M)*_{\pi_1(N)}\pi_1(M)$ on the group von Neumann algebra $\mathcal{N}(\pi_1(M)^{}/\mathrm{Im}(\pi_1(N))^{\pi_1(M)})$. One can define an $\eta$-invariant with values in the abelianisation of $\mathcal{N}(\pi_1(M)^{}/\mathrm{Im}(\pi_1(N))^{\pi_1(M)})$ and our $L^2$ $\rho$-invariant is then obtained from the latter $\eta$-invariant using the canonical trace of $\mathcal{N}(\pi_1(M)^{}/\mathrm{Im}(\pi_1(N))^{\pi_1(M)})$. Using a pair of finite-dimensional representations of  $\pi_1(M)$ agreeing on  $\pi_1(N)$, we can also define a relative version of the APS $\rho$-invariants and in this case our invariants seem to be stongly related to the invariants defined by Bunke in \cite{B15}, which has inspired our construction.

The fact that the relative $L^2$-$\rho$-invariant is an invariant of the bordism class of psc metrics is then established by a modification of a standard argument using the (higher) APS index formula. We later prove the important fact that the relative $L^2$-$\rho$-invariant can also be obtained as the $L^2$-$\eta$-invariant of a suitable cover of the double of the manifold with boundary. We make use of this characterisation of our invariant as an $L^2$-$\eta$-invariant to establish further fundamental properties of the relative $\rho$-invariant.

The paper is organised as follows. In the second section, we define in detail the relative versions of the groups appearing in the Stolz positive scalar curvature sequence, which again fit in a long exact sequence. In the third section, we define the relative $L^2$-$\rho$-invariant associated to Riemannian metrics which are collared near the boundary and prove that it is an invariant of the bordism class of a psc metric. This is followed by the characterisation of the relative $L^2$-$\rho$-invariant as an $L^2$-$\eta$-inavriant. In the fourth section we use the latter characterisation to establish the fundamental properties of our invariant. In the last section we use the developed theory and the techniques of \cite{PS07II} to obtain some geometric applications
\section*{Acknowledgements}
Part of the research for this paper was conducted during the stay of the first two authors in the \textit{Mathematisches Forschungsinstitut Oberwolfach}. The stay was funded through the Research in Pairs program.
\section{Relative Stolz psc sequence}
In this section we define the the relative version of the positive scalar curvature exact sequence of Stolz. For the definitions in the absolute case we refer the reader to \cite{PS14}*{Definition 1.26}\footnote{\cite{PS14}*{Definition 1.26.2} contains a typo. The group defined there is $\mathrm{R}^{\mathrm{Spin}}_{n+1}(X)$.} and \cite{RS01}*{Definition 5.4}. Let $\alpha\colon Y\to X$ be a continuous map of the topological spaces $Y$ and $X$. First we recall the definition of the relative spin bordism group.
\begin{definition}
$\Omega_n^{\mathrm{Spin}}(\alpha)$ will denote the usual relative spin bordism group. Cycles are given by tuples $(M, N,  F\colon M\to X, f\colon N \to Y)$ such that $M$ is an $n$-dimensional compact spin manifold with boundary $N$ and $F|_{N}=\alpha\circ f$.
This cycle is null-bordant if there exists $(V,E,e)$ where
	\begin{itemize}
		\item $V$ is an $n+1$-dimensional manifold with corner with $\partial V = M\cup \partial_iV$ and corner of $V$ is $\partial_cV = M\cap \partial_iV=N$.
		\item $E$ is a continuous map $V \rightarrow X$ with $E_{|M} = F$.
		\item $e$ is a continuous map $\partial_iV \rightarrow Y$ with $E_{|\partial_iV} = \alpha \circ e$ and $e_{|N} = f$.
	\end{itemize}
  The addition of two cycles is defined to be their disjoint union. The inverse of a cycle is obtained by reversing the spin structure. Two cycles are equivalent if their difference is null-bordant. The group $\Omega_n^{\mathrm{Spin}}(\alpha)$ is defined as the set of equivalence classes of the above cylces and the above addition induces the group operation. In the following, manifolds such as $V$ providing a cobordism between other manifolds will be called bordisms and $\partial_iV$ will be referred to as the internal boundary of the bordism $V$.
\end{definition}
\begin{definition}
The cycles for $\mathrm{Pos}^{\mathrm{Spin}}_{n}(\alpha)$ have the form
$$(M, N, F\colon M\to X, f\colon N\to Y, g),$$  where $(M, N, F, f)$ defines a cycle for $\Omega_n^{\mathrm{Spin}}(\alpha)$ and $g$ is a metric with positive scalar curvature on $M$ which is collared near $N$. 
A cycle is null-bordant if there exists $(V,E,e,G)$ where
	\begin{itemize}
		\item $(V,E,e)$ provides null-bordism for $(M, N, F,f)$ seen as a cycle in $\Omega_n^{\mathrm{Spin}}(\alpha)$.
		\item $G$ is a positive scalar curvature metric on $V$ which is collared near $\partial V$ and restricts to $g$ on $M$. We assume that $G$ has the form $G_{|\partial_cV} + \mathrm{d}t^2+ \mathrm{d}s^2$ in a bicollar neighbourhood of the corner $\partial_c V$.
	\end{itemize}

Addition and inversion of cycls are defined by taking the disjoint union and reversing the spin structure, respectively. Two cylces are called equivalent if their difference is null-bordant and the group $\mathrm{Pos}^{\mathrm{Spin}}_{n}(\alpha)$ is defined as the quotient of the set of cycles by this equivalence relation. Again, addition of cycles induces the group operation.
  \label{def:relpos}
\end{definition}
\begin{definition}
Cycles of $\mathrm{R}_{n+1}^{\mathrm{Spin}}(\alpha)$ are given by tuples $$(Z,M,W,H\colon Z\to X,F\colon M\to X, h\colon W\to Y, g)$$ where 
\begin{itemize}
\item $M$ and $W$ are compact $n$-dimensional spin manifolds with $N\coloneqq \partial M = \partial W$ and $(M,N,F,f\coloneqq h|_N,g)$ defines a cycle of $\mathrm{Pos}^{\mathrm{Spin}}_n(\alpha)$.
\item $(Z,H,h)$ provides an equivalence of $(M, N, F, f)$ with the empty set, seen as cycles of $\mathrm{\Omega}^{\mathrm{Spin}}_n(\alpha)$, in particular we are assuming $Z$ to be manifold with corner with $\partial Z= W\cup_NM$ and corner $\partial_c Z = N$.
\end{itemize}

This tuple is null-bordant if 
	\begin{itemize}
		\item There exists $(V,E,e,G)$ providing an equivalence between $(M, N, F, f\coloneqq h|_{N}, g)$ the empty tuple seen as cycles for $\mathrm{Pos}^{\mathrm{Spin}}_n(\alpha)$. We note that $\partial_i V$ provides an equivalence between $(N,f, g_{|N})$ and the empty tuple considered as cycles of $\mathrm{Pos}^{\mathrm{Spin}}_{n-1}(Y)$.
		\item There exists $(U,L\colon U\rightarrow Y)$ where
		\begin{itemize}
			\item the spin manifold with corner $U$ is a null-bordism for $W$,
			\item $\partial_i U = \partial _i V$ and
			\item the map $L$ extends the maps $H$.
		\end{itemize}
		In particular, $(U,L,G|_{\partial_iV})$ provides an equivalence between $(W,h,g_{|N})$ and the empty tuple considered as cycles in $\mathrm{R}^{\mathrm{Spin}}_{n}(Y)$.
		\item There exists an $n+1$-dimensional spin manifold with corner $V^\prime$ with $\partial V^\prime = V \cup_{\partial_i V} U$ and a continuous map $E^\prime \colon V^\prime\to X$ extending $H$ satisfying $ E^\prime|_U = \alpha \circ L$. 
	\end{itemize}
Addition and inversion of cycles are defined by taking the disjoint union and reversing the spin structure, respectively. Two cycles are called equivalent if their difference is null-bordant and the group $\mathrm{R}_{n+1}^{\mathrm{Spin}}(\alpha)$ is defined as the quotient of the set of cycles by this equivalence relation. Again, addition of cycles induces the group operation.
\end{definition}

Given maps $\beta\colon C \rightarrow D$, $s\colon C \rightarrow Y$ and $t\colon D \rightarrow X$ which make the diagram
$$\begin{tikzcd}
C \arrow{r}{s}\arrow{d}{\beta}& Y \arrow{d}{\alpha} \\
D \arrow{r}{t} & X
\end{tikzcd}
$$
commutative, one obtains group homomorphisms
$$\Omega_n^{\mathrm{Spin}}(\beta) \rightarrow \Omega_n^{\mathrm{Spin}}(\alpha)\,, \quad\mathrm{Pos}_n^{\mathrm{Spin}}(\beta) \rightarrow \mathrm{Pos}_n^{\mathrm{Spin}}(\alpha)\quad\mbox{and} \quad\mathrm{R}_n^{\mathrm{Spin}}(\beta) \rightarrow \mathrm{R}_n^{\mathrm{Spin}}(\alpha),$$ by postcomposing the maps in the definition of cycles for these groups with $s$ and $t$ suitably. It is then easy to see that the relative $\Omega^{\mathrm{Spin}}_{n}$, $\mathrm{Pos}^{\mathrm{Spin}}_{n}$ and $\mathrm{R}^{\mathrm{Spin}}_ {n}$ groups defined above form functors on the category whose objects are continuous maps of topological spaces and where morphisms are given by pairs $(t,s)$ as above.

	First, observe that the relative groups associated to $\alpha\colon \emptyset\to X$ coincide with the absolute groups associated to $X$. Using this obsevation and the commutativity of the diagram
		\[
		\xymatrix{\emptyset\ar[d]\ar[r]& \emptyset\ar[d]\ar[r]& Y\ar[d]^\alpha\\Y\ar[r]^\alpha& X\ar[r]^{id_X}& X}
		\]
	we obtain maps
	$$\alpha_*: \Omega^{\mathrm{Spin}}_{n+1}(Y) \rightarrow \Omega^{\mathrm{Spin}}_{n+1}(X)\quad\mbox{and} \quad (\mathrm{id}_X)_*:\Omega^{\mathrm{Spin}}_{n+1}(X)\rightarrow \Omega^{\mathrm{Spin}}_{n+1}(\alpha).$$
Moreover, the following maps are well defined:
\begin{itemize}
	\item the boundary map $\mathrm{R}^{\mathrm{Spin}}_{n+1}(\alpha)\rightarrow \mathrm{Pos}^{\mathrm{Spin}}_{n}(\alpha)$ which sends the tuple given by 
	$(Z,M,W,H\colon Z\to X,F\colon M\to X, h\colon W\to Y, g)$ to $(M,F\colon M\to X, h_{|N}\colon N\to Y, g)$;
	\item the forgetful map $\mathrm{Pos}^{\mathrm{Spin}}_{n}(\alpha)\rightarrow \Omega^{\mathrm{Spin}}_n(\alpha)$ which sends the cycle given by  $(M, N, F\colon M\to X, f\colon N\to Y, g)$ to $(M, N, F\colon M\to X, f\colon N\to Y)$;
		\item the map $\Omega^{\mathrm{Spin}}_n(\alpha)\rightarrow \mathrm{R}^{\mathrm{Spin}}_{n}(\alpha)$ which sends
		$(M, N, F\colon M\to X, f\colon N\to Y)$ to the cycle $(M,\emptyset, N, F\colon M\to X,\emptyset, f\colon N\to Y, \emptyset)$.
\end{itemize}
	The following two propositions follow immediately from the definitions.
	\begin{proposition}\label{les-pair}
		There exist a long exact sequence
	\[
	\xymatrix{\cdots\ar[r]& \Omega^{\mathrm{Spin}}_{n+1}(Y)\ar[r]^{\alpha_*}&\Omega^{\mathrm{Spin}}_{n+1}(X)\ar[r]^{(\mathrm{id}_X)_*}& \Omega^{\mathrm{Spin}}_{n+1}(\alpha)\ar[r]^{\partial}& \Omega^{\mathrm{Spin}}_{n}(Y)\ar[r]&\cdots}
	\]
	where $\partial$ is given at the level of cycles by sending $(M, N, F\colon M\to X, f\colon N\to Y)$ to $(f\colon N\to Y)$. Similar long exact sequences exist for the $\mathrm{Pos}^{\mathrm{Spin}}_{*}$ and $\mathrm{R}^{\mathrm{Spin}}_ {*}$ functors. The long exact sequence for  $\mathrm{Pos}^{\mathrm{Spin}}_{*}$ terminates at  $\mathrm{Pos}^{\mathrm{Spin}}_{2}(\alpha)$ and the long exact sequence for $\mathrm{R}^{\mathrm{Spin}}_ {*}$ terminates at $\mathrm{R}^{\mathrm{Spin}}_ {3}(\alpha)$.
	
		\end{proposition}

 \begin{proposition}\label{relative-Stolz}
 	We have a \emph{relative} Stolz long exact sequence
 	\begin{equation}
 	\cdots \rightarrow \mathrm{R}^{\mathrm{Spin}}_{n+1}(\alpha)\rightarrow \mathrm{Pos}^{\mathrm{Spin}}_{n}(\alpha)\rightarrow \Omega^{\mathrm{Spin}}_n(\alpha)\rightarrow \mathrm{R}^{\mathrm{Spin}}_{n}(\alpha)\rightarrow \cdots
 	\end{equation}
 	Furthermore, the arrows in this exact sequence are natural transformation of functors.
      \end{proposition}
      
\begin{corollary}
	If $\alpha\colon Y\to X$ is 2-connected, then $\mathrm{Pos}^{\mathrm{Spin}}_{n}(\alpha)\cong \Omega^{\mathrm{Spin}}_n(\alpha)$ for all $n\geq5$.
\end{corollary}
\begin{proof}
	By \cite{Stolz}, if $\alpha$ is 2-connected then $\mathrm{R}^{\mathrm{Spin}}_{n}(Y)\cong \mathrm{R}^{\mathrm{Spin}}_{n}(X)$ for all $n\geq5$. Then, by Proposition \ref{les-pair} applied to the $\mathrm{R}^{\mathrm{Spin}}_{*}$ groups, we have that $\mathrm{R}^{\mathrm{Spin}}_{n}(\alpha)$ is trivial for all $n \geq 5$. Finally, by Proposition \ref{relative-Stolz}, we obtain the wished result.
\end{proof}

\section{Index-theoretic preliminaries}
In this section, we recall some background material from higher index theory on manifolds with boundary. For a more general, yet concise review of related material we refer the reader to \cite{PS07}*{Sections 2 \& 3}.

Let $W$ be a compact spin manifold with boundary $X$. Let $\mathcal{A}$ be a $C^*$-algebra and $\mathcal{L}$ be a flat bundle of finitely generated projective Hilbert $C^*$-modules over $\mathcal{A}$. Denote by $\slashed{D}_W$ the spin Dirac operator of $W$. This is a first order differential operator acting on sections of the spinor bundle $\slashed{S}_W$ on $W$. Twisting $\slashed{D}_W$ by the bundle $\mathcal{L}$ we obtain an $\mathcal{A}$-linear differential operator $\slashed{D}_{W,\mathcal{L}}$ acting on the smooth sections of $\slashed{S}\otimes \mathcal{L}$. In what follows, we will suppose that the Riemannian metric of $W$ is collared near the boundary and has positive scalar curvature there. Then one can define a higher index $\ind(\slashed{D}_{W,\mathcal{L}})$ for the operator $\slashed{D}_{W,\mathcal{L}}$ in $K_{\mathrm{dim}W}(\mathcal{A})$.
\begin{remark}
A priori, $\ind(\slashed{D}_{W,\mathcal{L}})$ depends on the positive scalar curvature metric on the boundary. If the metric has positive scalar curvature everywhere, then $\ind(\slashed{D}_{W,\mathcal{L}})$ vanishes. We will need this fact in the proof of the well-definedness of the relative $L^2$-$\rho$-invariant.
\end{remark}
Denote by $\overline{[\mathcal{A},\mathcal{A}]}$ the closure of the linear subspace of $\mathcal{A}$ generated by elements of the form $ab-ba$ with $a,b \in \mathcal{A}$. The quotient $\mathcal{A}/\overline{[\mathcal{A},\mathcal{A}]}$ is called the abelianisation of $\mathcal{A}$ and will be denoted by $\mathcal{A}_{\mathrm{ab}}$. The canonical projection map $\mathcal{A} \rightarrow \mathcal{A}_{\mathrm{ab}}$ then defines a trace on $\mathcal{A}$. We will refer to this map the algebraic trace and denote it by $\mathrm{tr}^{\mathrm{alg}}$. Note that any trace on $\mathcal{A}$ factors uniquely through the algebraic trace. The algebraic trace then induces a map
$$K_0(\mathcal{A}) \rightarrow \mathcal{A}_{\mathrm{ab}},$$
which we will also denote by $\mathrm{tr}^{\mathrm{alg}}$.

\begin{definition}
Suppose that the manifold $W$ is even-dimensional. The $\mathcal{A}_{\mathrm{ab}}$-valued index of the operator $\slashed{D}_{W,\mathcal{L}}$ is defined to be the image of $\ind(\slashed{D}_{W,\mathcal{L}})$ under the algebraic trace map and will be denoted by $\ind_{[0]}(\slashed{D}_{W,\mathcal{L}})$.
\end{definition}
We will now recall a higher APS index formula for the $\mathcal{A}_{\mathrm{ab}}$-valued index. First, we define the $\mathcal{A}_{\mathrm{ab}}$-valued $\eta$-invariant of the twisted Dirac operator on the boundary. Denote by $\slashed{D}_X$ the spin Dirac operator on $X$ and by $\slashed{D}_{X,\mathcal{L}}$ the Dirac operator twisted with the restriction of $\mathcal{L}$ to $X$. The operator $\slashed{D}_{X,\mathcal{L}}e^{-t\slashed{D}_{X,\mathcal{L}}^2}$ is then a smoothing operator in the Mischenko-Fomenko calculus for each $t>0$. Let $K_t$ denote the kernel of the latter operator. For each $x \in X$, $K_t(x,x)$ can then be seen as an $\mathcal{A}$-linear endomorphism of $(\slashed{S}_X\otimes\mathcal{L})_x$ and thus we can take its algebraic trace, which we will denote by $\mathrm{tr}^{\mathrm{alg}}(K_t(x,x))$. See \cite{Sch05}*{Proposition 2.26 \& Lemma 2.29} for more details.
\begin{definition}
\begin{itemize}
\item Define $\mathrm{TR}$ by the following integral
$$\mathrm{TR}(\slashed{D}_{X,\mathcal{L}}e^{-t\slashed{D}_{X,\mathcal{L}}^2}):=\int_X\mathrm{tr}^{\mathrm{alg}}(K_t(x,x))\mathrm{d}x \in\mathcal{A}_{ab}.$$

\item The $\mathcal{A}_{\mathrm{ab}}$-valued $\eta$-invariant of $\slashed{D}_{X,\mathcal{L}}$ is defined by the integral
$$\eta_{[0]}(\slashed{D}_{X,\mathcal{L}}):=\frac{1}{\sqrt{\pi}}\int_0^\infty\mathrm{TR}(\slashed{D}_{X,\mathcal{L}}e^{-t\slashed{D}_{X,\mathcal{L}}^2})\frac{\mathrm{d}t}{\sqrt{t}}\,\in\mathcal{A}_{ab}.$$
\end{itemize}
\end{definition}
\begin{remark}
The above definition makes sense even if we forget about $W$ and start with a flat bundle $\mathcal{L}$ of finitely generated projective $\mathcal{A}$-modules on a closed odd-dimensional spin manifold $X$ which has positive scalar curvature.
\end{remark}
\begin{theorem}[\cite{PS07}*{Theorem 2.24}]
The $\mathcal{A}_{\mathrm{ab}}$-valued index $ind_{[0]}(\slashed{D}_{W,\mathcal{L}})$ is given by
$$\int_W\widehat{A}(W)(x)\wedge\mathrm{ch}(\mathcal{L})(x)\mathrm{d}x -\frac{1}{2}\eta_{[0]}(\slashed{D}_{X,\mathcal{L}}).$$
Here, $\widehat{A}(W)(x)$ denotes the value of the differential form constructed using Chern-Weil theory representing the $\widehat{A}$-class of the manifold $W$ and $\mathrm{ch}(\mathcal{L})$ denotes the $\mathcal{A}_{\mathrm{ab}}$-valued differential form representing the Chern character of the bundle $\mathcal{L}$ over $X$ (\cite{Sch05}*{Definition 4.1}).
\label{thm:higherAPS}
\end{theorem}
\section{The relative $L^2$-$\rho$-invariant}\label{sec:defrelrho}
Let $\Lambda$ and $\Gamma$ be discrete groups and let $\varphi\colon\Lambda \rightarrow \Gamma$ be a group homomorphism. $\varphi$ induces a map $\mathrm{B}\varphi\colon\mathrm{B}\Lambda \rightarrow \mathrm{B}\Gamma$, which can be assumed to be injective\footnote{Replace $\mathrm{B}\Gamma$ by the mapping cylinder of $\mathrm{B}\varphi$.}. We will construct a group homomorphism
$$\rho^{\varphi}_{2}\colon\mathrm{Pos}_n^{\mathrm{Spin}}(\mathrm{B}\varphi) \rightarrow \reals$$
for all odd $n$ using $\eta$-invariants of suitable Dirac operators.

Denote by $\mathrm{Im}(\varphi)^\Gamma$ the normal subgroup generated by the image of $\Lambda$ in $\Gamma$ and by $\mathcal{N}(\Gamma^{}/\mathrm{Im}(\varphi)^\Gamma)$ the von Neumann algebra of the group $\Gamma^{}/\mathrm{Im}(\varphi)^\Gamma$. Denote by $\pi_r\colon\Gamma \rightarrow \Gamma^{}/\mathrm{Im}(\varphi)^\Gamma$ the canonical projection map and by $\pi_{\mathrm{triv}}\colon \Gamma \rightarrow \Gamma^{}/\mathrm{Im}(\varphi)^\Gamma$ the trivial group homomorphism. The homomorphisms $\pi_r$ and $\pi_{\mathrm{triv}}$ give rise to a homomorphism $\pi_{r}*\pi_{\mathrm{triv}}\colon \Gamma *_{\Lambda}\Gamma \rightarrow \Gamma^{}/\mathrm{Im}(\varphi)^\Gamma$ by the universal property of the amalgamated product. Let $\Pi$ be the composition
\begin{equation}\label{Pi}\Gamma *_{\Lambda}\Gamma \xrightarrow{\pi_{r}*\pi_{\mathrm{triv}}} \frac{\Gamma}{\mathrm{Im}(\varphi)^\Gamma} \hookrightarrow \mathcal{N}\left(\frac{\Gamma}{\mathrm{Im}(\varphi)^\Gamma}\right),\end{equation}
where the second arrow is the inclusion of $\Gamma^{}/\mathrm{Im}(\varphi)^\Gamma$ into its von Neumann algebra. There are two natural traces on $\mathcal{N}(\Gamma^{}/\mathrm{Im}(\varphi)^\Gamma)$: the algebraic trace
$$\mathrm{tr}^{\mathrm{alg}}\colon \mathcal{N}\left(\frac{\Gamma}{\mathrm{Im}(\varphi)^\Gamma}\right) \rightarrow \mathcal{N}\left(\frac{\Gamma}{\mathrm{Im}(\varphi)^\Gamma}\right)_{\mathrm{ab}}$$
and the canonical trace
$$\tau_{\mathrm{can}}\colon \mathcal{N}\left(\frac{\Gamma}{\mathrm{Im}(\varphi)^\Gamma}\right) \rightarrow \mathbb{C}$$
which is the continuous extension of the canonical trace $\mathbb{C}(\Gamma^{}/\mathrm{Im}(\varphi)^\Gamma) \rightarrow \mathbb{C}$ on the group ring, mapping an element to the coefficient of the neutral element. Note that the canonical trace factorises through the universal trace: $\tau_{\mathrm{can}} = \tau^\prime_{\mathrm{can}} \circ \mathrm{tr}^{\mathrm{alg}}$ with $\tau^\prime_{\mathrm{can}}$ a unique linear map $\mathcal{N}(\Gamma^{}/\mathrm{Im}(\varphi)^\Gamma)_{\mathrm{ab}} \rightarrow \mathbb{C}$.

In the following, let $n \in \mathbb{N}$ be odd. A class in $\mathrm{Pos}_n^{\mathrm{Spin}}(\mathrm{B}\varphi)$ is represented by a tuple $(M,N,c,g)$ with $M$ a compact $n$-dimensional spin manifold with boundary $N$, endowed with a positive scalar curvature metric $g$ which is collared near the boundary and a continuous map $c\colon (M,N) \rightarrow (\mathrm{B}\Gamma,\mathrm{B}\Lambda)$\footnote{We are using the assumption that the map $\mathrm{B}\varphi$ is injective.}. In the notation of Definition~\ref{def:relpos}, this cycle would be denoted by $(M,N,c,c|_N,g)$. Denote by $\mathrm{D}M$ the double $M \cup_N (-M)$ of $M$, where $-M$ denotes the manifold $M$ with the reversed spin structure. 
Since the metric of $M$ is collared near the boundary, we naturally obtain a Riemannian metric on the double with positive scalar curvature.  The map $c$ then induces a map
$$\mathrm{D}c\colon\mathrm{D}M \rightarrow \mathrm{B}\Gamma\cup_{\mathrm{B}\Lambda}\mathrm{B}\Gamma \rightarrow \mathrm{B}(\Gamma *_{\Lambda}\Gamma).$$
Denote by $\widetilde{\mathrm{D}M}$ the cover associated to the map $\mathrm{D}c$, Consider the bundle of von Neumann algebras
$$\mathcal{L} \coloneqq \widetilde{\mathrm{D}M} \times_{\Gamma *_{\Lambda}\Gamma}\mathcal{N}\left(\frac{\Gamma}{\mathrm{Im}(\varphi)^\Gamma}\right)$$
where an element $\Gamma *_{\Lambda}\Gamma$ acts on $\mathcal{N}(\Gamma^{}/\mathrm{Im}(\varphi)^\Gamma)$ via left multiplication by its image under the map $\Pi$ from \eqref{Pi}. Let $\slashed{D}_{\mathrm{D}M,\mathcal{L}}$ be the spin Dirac operator of $\mathrm{D}M$ twisted by $\mathcal{L}$. Finally we can define the relative $L^2$-$\rho$-invariant.
\begin{definition}
The relative $L^2$-$\rho$-invariant of a cycle $(M,N,c,g)$ representing a class in $\mathrm{Pos}_n^{\mathrm{Spin}}(\mathrm{B}\varphi)$ is
$$\rho^{\varphi}_{(2)}(M,N,c,g) \coloneqq \tau^\prime_{\mathrm{can}}(\eta_{[0]}(\slashed{D}_{\mathrm{D}M,\mathcal{L}})).$$
\end{definition}
\begin{proposition}
The relative $L^2$-$\rho$-invariant descends to a group homomorphism
 $$\rho^{\varphi}_{(2)}\colon \mathrm{Pos}_n^{\mathrm{Spin}}(\mathrm{B}\varphi) \rightarrow \reals.$$

\end{proposition}
\begin{proof}
We need to show that if $(M,N,c,g)$ and $(M^\prime,N^\prime,c^\prime,g^\prime)$ represent the same class in $\mathrm{Pos}_n^{\mathrm{Spin}}(\mathrm{B}\varphi)$, then
$$\rho^{\varphi}_{(2)}(M,N,c,g) = \rho^{\varphi}_{(2)}(M^\prime,N^\prime,c^\prime,g^\prime).$$
This is equivalent to showing that if $(M,N,c,g)$ is null-bordant, namely represents zero in $\mathrm{Pos}^{\mathrm{Spin}}_{n}(\mathrm{B}\varphi)$, then $\rho^{\varphi}_{(2)}(M,N,c,g) = 0$.

Let $(M,N,c,g)$ be null-bordant. This means, there exists a tuple $(V,b,b|_N,G)$ as in Definition \ref{def:relpos} providing an equivalence between $(M,N,c,c|_N,g)$ and the empty set.
Consider the manifold $W\coloneqq V \cup_{\partial_i V} -V$, which is a compact spin manifold with boundary $\mathrm{D}M$. 
Since $G$ is collared near the boundary and bicollared near the corner of $V$, we obtain naturally a psc metric on $W$ which is collared near the boundary. On the boundary this metric coincides with the psc metric on $\mathrm{D}M$ defined using $g$. Furthermore, using $b$ we get a map $\mathrm{D}b\colon W \rightarrow \mathrm{B}\Gamma\cup_{\mathrm{B}\Lambda}\mathrm{B}\Gamma \rightarrow \mathrm{B}(\Gamma *_{\Lambda}\Gamma)$ making the diagram
$$\begin{tikzcd}
    W \arrow{rd}{\mathrm{D}b} & \\
    DM \arrow{u} \arrow{r}{\mathrm{D}c} & \mathrm{B}(\Gamma *_{\Lambda}\Gamma)
   \end{tikzcd}
$$
commutative. Let $\widetilde{W}$ be the cover associated to $Db$ and set
$$\mathcal{L}_{W} \coloneqq \widetilde{W} \times_{\Gamma *_{\Lambda}\Gamma}\mathcal{N}\left(\frac{\Gamma}{\mathrm{Im}(\varphi)^\Gamma}\right)$$
where the action of $\Gamma *_{\Lambda}\Gamma$ on $\mathcal{N}(\Gamma^{}/\mathrm{Im}(\varphi)^\Gamma)$ is, as above, induced by $\Pi$. Clearly, $\mathcal{L}_{W}$ restricts to $\mathcal{L}$ on $DM$. Since the metric on $\mathrm{D}M$ has positive scalar curvature and $W$ is even-dimensional, one can define an index $\ind(D_{W,\mathcal{L}}) \in K_0(\mathcal{N}(\Gamma/\mathrm{Im}(\varphi)^\Gamma))$. By Theorem \ref{thm:higherAPS}, we have
$$\mathrm{tr}^{\mathrm{alg}}(\ind(D_{W,\mathcal{L}})) = \ind_{[0]}(D_{W,\mathcal{L}}) = \int_{W}\widehat{A}(W)(X)\wedge \mathrm{ch}(\mathcal{L})(x)\mathrm{d}x -\frac{1}{2}\eta_{[0]}(\slashed{D}_{\mathrm{D}M,\mathcal{L}}).$$
Since $\mathcal{L}$ is a flat bundle with each fibre a free $\mathcal{N}(\Gamma^{}/\mathrm{Im}(\varphi)^\Gamma)$-module of rank $1$, we have $\mathrm{ch}(\mathcal{L})(x) = 1 \in \mathcal{N}(\Gamma^{}/\mathrm{Im}(\varphi)^\Gamma)_{\mathrm{ab}}$. Hence,
$$\ind_{[0]}(D_{W,\mathcal{L}}) = \int_{W}\widehat{A}(W)(X)\mathrm{d}x\cdot1 -\frac{1}{2}\eta_{[0]}(\slashed{D}_{\mathrm{D}M,\mathcal{L}}).$$
On the other hand,
$$\int_{W}\widehat{A}(W) = \int_{V}\widehat{A}(W) + \int_{-V}\widehat{A}(W) = 0.$$
The metric on $W$ has positive scalar curvature; therefore, $\ind(D_{W,\mathcal{L}})$ and hence $\ind_{[0]}(D_{W,\mathcal{L}})$ vanish. This implies that
$$\eta_{[0]}(\slashed{D}_{\mathrm{D}M,\mathcal{L}}) = 0$$
Applying $\tau^\prime_{\mathrm{can}}$ to both sides gives
$$\rho^{\varphi}_{(2)}(M,N,c,g) = 0.$$
Thus we have shown that if $(M,N,c,g)$ represents the zero class in $\mathrm{Pos}^{\mathrm{Spin}}_n(\mathrm{B}\varphi)$ then $\rho^{\varphi}_{(2)}(M,N,c,g) = 0$, which proves that $\rho^{\varphi}_{(2)}$ descends to a map
$$\mathrm{Pos}_n^{\mathrm{Spin}}(\mathrm{B}\varphi) \rightarrow \reals.$$
Since addition in $\mathrm{Pos}_n^{\mathrm{Spin}}(\mathrm{B}\varphi)$ is induced by taking disjoint unions, this map is clearly a group homomorphism.
\end{proof}
\begin{remark}
An element $\gamma \in \Gamma^{}/\mathrm{Im}(\Lambda)^\Gamma$ whose conjugacy class has polynomial growth gives rise to a trace $K_0(\mathcal{N}(\Gamma^{}/\mathrm{Im}(\Lambda)^\Gamma)) \rightarrow \mathbb{C}$. We can then define a relative version of John Lott's delocalised  $\eta$-invariant (see \cite{L92,L92II}) by applying this trace to $\eta_{[0]}(\slashed{D}_{\mathrm{D}M,\mathcal{L}})$. The techniques used below, combined with the approach of \cite{PS07II}, should give an analogue of Theorem \ref{thm:geometricresult} for $4k+1$-dimensional manifolds without much difficulty.
\end{remark}
\begin{remark}
A pair of finite dimensional representations of $\Gamma$ agreeing on $\Lambda$ gives rise to a representation of $\Gamma *_{\Lambda}\Gamma$, which can be used to define a finite dimensional flat bundle on $\mathrm{D}M$ (see above for the notation). The $\eta$-invariant of the spin Dirac operator of $\mathrm{D}M$ twisted with the latter bundle can be shown to be an invariant of the bordism class of the psc metric on $M$ (using an argument similar to the one used in the proof of the previous theorem). This can be seen as a generalisation of the APS $\rho$-invariant associated to a pair of representations for manifolds with boundary and seems to be strongly related to the invariants considered by Bunke in \cite{B15}, which have inspired our constuction.
\end{remark}
Now we show that the relative $L^2$-$\rho$-invariant is equal to the $L^2$ $\eta$-invariant of a suitable $\Gamma^{}/\mathrm{Im}(\varphi)^\Gamma$-cover of $\mathrm{D}M$.

We first recall one definition of the $L^2$-$\eta$-invariant. Let $X$ be a closed odd-dimensional spin manifold and $\widetilde{X} \rightarrow X$ be a $G$-cover of $X$ for some discrete group $G$. Denote by $\widetilde{\slashed{D}}$ the Dirac operator on $\widetilde{X}$, constructed using the pullback of the metric on $X$, and by $k^\prime_t$ the Schwartz kernel of the operator $\widetilde{\slashed{D}} e^{-t(\widetilde{\slashed{D}})^2}$. Then the $L^2$-$\eta$-invariant of the spin Dirac operator associated to the cover $\widetilde{X} \rightarrow X$ is given by
\begin{equation}\label{eta:formula}\eta_{(2)}(\widetilde{\slashed{D}}) \coloneqq \frac{1}{\sqrt{\pi}}\int_0^\infty\left(\int_\mathcal{F}k_t(x,x)\mathrm{d}x\right)\frac{\mathrm{d}t}{\sqrt{t}}\end{equation}
where $\mathcal{F}$ is any fundamental domain of the action of $G$ on $\widetilde{X}$.
The homomorphism $\pi_r*\pi_{\mathrm{triv}}$ gives rise to a map $\mathrm{B}(\pi_r*\pi_{\mathrm{triv}})\colon \mathrm{B}(\Gamma*_\Lambda\Gamma) \rightarrow \mathrm{B}\left(\frac{\Gamma}{\mathrm{Im}(\varphi)^\Gamma}\right)$ of classifying spaces. Denote by $(\mathrm{D}c)^\prime$ the composition
\begin{equation}\label{Dc'}\mathrm{D}M \xrightarrow{\mathrm{D}c} \mathrm{B}(\Gamma*_\Lambda\Gamma) \xrightarrow{\mathrm{B}(\pi_r*\pi_{\mathrm{triv}})} \mathrm{B}\left(\frac{\Gamma}{\mathrm{Im}(\varphi)^\Gamma}\right).\end{equation}
\begin{proposition} The relative $L^2$-$\rho$-invariant
$\rho_2(M,N,g,c)$ is given by the $L^2$-$\eta$-invariant of the cover of $\mathrm{D}M$ associated to the classifying map $(\mathrm{D}c)^\prime$.
\label{prop:relrhol2eta}
\end{proposition}
\begin{proof}
Denote by $(\mathrm{D}M)^\prime$ the cover associated to $(\mathrm{D}c)^\prime$. Thus, we have $(\mathrm{D}M)^\prime = \widetilde{\mathrm{D}M}/\mathrm{Ker}(\pi_r*\pi_{\mathrm{triv}})$. The map
$$\widetilde{\mathrm{D}M} \times \mathcal{N}\left(\frac{\Gamma}{\mathrm{Im}(\varphi)^\Gamma}\right) \rightarrow (\mathrm{D}c)^\prime \times \mathcal{N}\left(\frac{\Gamma}{\mathrm{Im}(\varphi)^\Gamma}\right)$$
which is the natural projection on the first component and the identity on the second component induces an isomorphism of the bundles $\mathcal{L}=\widetilde{\mathrm{D}M} \times_{\Gamma *_{\Lambda}\Gamma}\mathcal{N}\left(\frac{\Gamma}{\mathrm{Im}(\varphi)^\Gamma}\right)$
and $(\mathrm{D}M)^\prime \times_{\frac{\Gamma}{\mathrm{Im}(\varphi)^\Gamma}} \mathcal{N}\left(\frac{\Gamma}{\mathrm{Im}(\varphi)^\Gamma}\right)$.
The $\eta$-invariant of the Dirac operator on $\mathrm{D}M$ twisted with the bundle $(\mathrm{D}M)^\prime \times_{\frac{\Gamma}{\mathrm{Im}(\varphi)^\Gamma}} \mathcal{N}\left(\frac{\Gamma}{\mathrm{Im}(\Lambda)^\Gamma}\right)$ is the $L^2$-$\eta$-invariant of the cover $(\mathrm{D}M)^\prime$ (\cite{PS07}*{Proposition E.11}).
\end{proof}

\section{Basic properties of the relative $L^2$-$\rho$-invariant}
In this section, we relate the relative $L^2$-$\rho$-invariant to the classical $L^2$-$\rho$-invariant and establish some of its fundamental properties.

The classical $L^2$-$\rho$-invariant gives rise to a group homomorphism
$$\rho^{G}_{(2)}\colon \mathrm{Pos}^{\mathrm{Spin}}_n(\mathrm{B}G) \rightarrow \reals,$$
for odd $n$. A cycle for $\mathrm{Pos}^{\mathrm{Spin}}_n(\mathrm{B}G)$ is a tuple $(M,c,g)$ with $M$ a closed $n$-dimensional spin manifold, $g$ a positive scalar curvature metric on $M$ and $c$ a continuous map $M \rightarrow \mathrm{B}G$. Denote by $\widetilde{M}$ the $G$-cover associated to $c$ and by $\widetilde{\slashed{D}}$ and $\slashed{D}$ the spin Dirac operator on $\widetilde{M}$ and $M$ respectively. The classical $L^2$-$\rho$-invariant of $[(M,c,g)]$ is then defined to be
\begin{equation}\label{classical-rho}\rho_{(2)}^G[(M,c,g)]:=\eta_{(2)}(\widetilde{\slashed{D}}) - \eta(\slashed{D}),\end{equation} where $\eta(\slashed{D})$ is the usual $\eta$-invariant of $\slashed{D}$ introduced in \cite{APS75I}. Let $H$ be a further discrete group and $f\colon  H\rightarrow G$ a group homomorphism. Then there is a natural map
$$\mathrm{Pos}^{\mathrm{Spin}}_n(\mathrm{B}G) \rightarrow \mathrm{Pos}^{\mathrm{Spin}}_n(\mathrm{B}f)$$
induced by sending a cycle $(M,c,g)$ for $\mathrm{Pos}^{\mathrm{Spin}}_n(\mathrm{B}G)$ to the cycle $(M,\varnothing,c,\varnothing \rightarrow \mathrm{B}H,g)$.
\begin{remark}
We use the notation from the above discussion. There is a natural map $\Omega^{\mathrm{Spin}}_n(\mathrm{B}G) \rightarrow \Omega^{\mathrm{Spin}}_n(\mathrm{B}f)$ induced by sending a cycle $(M,c)$ for $\Omega^{\mathrm{Spin}}_n(\mathrm{B}G)$ to the cycle $(M,\varnothing,c,\varnothing \rightarrow \mathrm{B}H)$. The diagram
$$
\begin{tikzcd}
\mathrm{Pos}^{\mathrm{Spin}}_n(\mathrm{B}G) \arrow{r} \arrow{d} & \Omega^{\mathrm{Spin}}_n(\mathrm{B}G) \arrow{d} \\
\mathrm{Pos}^{\mathrm{Spin}}_n(\mathrm{B}f) \arrow{r} & \Omega^{\mathrm{Spin}}_n(\mathrm{B}f)
\end{tikzcd}
$$
is commutative.
\label{rem:posomegacommu}
\end{remark}
\begin{proposition}
Let $G$ be a discrete group and let $f\colon 1 \rightarrow G$ be the trivial homomorphism from the trivial group into $G$. Then the diagram
$$\begin{tikzcd}
\mathrm{Pos}^{\mathrm{Spin}}_n(\mathrm{B}G) \arrow{rd}[swap]{\rho^{G}_{(2)}} \arrow{r} & \mathrm{Pos}^{\mathrm{Spin}}_n(\mathrm{B}f) \arrow{d}{\rho^{f}_{(2)}} \\
 & \mathbb{R}
 \end{tikzcd}
$$
is commutative.
\label{prop:Gfcommu}
\end{proposition}
\begin{proof}
Let $(M,c,g)$ be a cycle for $\mathrm{Pos}^{\mathrm{Spin}}_n(\mathrm{B}G)$. We use the notation from the previous section. $\mathrm{D}M$ is then just the disjoint union $M \sqcup -M$. Since $\mathrm{B}1$ can be taken to be a point, we have $\mathrm{B}G\cup_{\mathrm{B}1}\mathrm{B}G = \mathrm{B}G \vee\mathrm{B}G$, the wedge sum. Analogously to \eqref{Dc'}, the map 
$$\mathrm{B}G \vee\mathrm{B}G \rightarrow \mathrm{B}(G*G) \rightarrow \mathrm{B}G$$  induced by mapping the first component of $G*G$ via the identity and the second component trivially to $G$ is the identity on the first copy of $\mathrm{B}G$ and crushes the second copy to the base point $\mathrm{B}1$. Thus the map $(\mathrm{D}c)^\prime\colon \mathrm{D}M = M \sqcup -M \rightarrow \mathrm{B}G$ restricted to $M$ is $c$ and restricted to $-M$ is the map to a point. Hence, the cover of $\mathrm{D}M$ associated to $(\mathrm{D}c)^\prime$ is the disjoint union of the $G$-cover of $M$ associated to $c$ and the trivial $G$-cover of $-M$. The $L^2$-$\eta$-invariant associated to this cover is thus the difference of the $L^2$-$\eta$-invariant associated to the cover classified by $c$ and the classical $\eta$-invariant of $M$. By Proposition \ref{prop:relrhol2eta}, this is the same as $\rho^{f}_{(2)}[(M,\varnothing,c,\varnothing \rightarrow \mathrm{B}H,g)]$. On the other hand, by \eqref{classical-rho}, the latter difference is simply the definition of $\rho^{G}_{(2)}[(M,c,g)]$. The claim follows.
\end{proof}
As above, let $G$ and $H$ be a pair of discrete groups and let $f\colon H \rightarrow G$ be a group homomorphism. Furthermore, suppose there exist group homomorphisms $s\colon H\rightarrow \Lambda$ and $t\colon G \rightarrow \Gamma$ making the diagram
$$
\begin{tikzcd}
H \arrow{r}{s}\arrow{d}{f} & \Lambda \arrow{d}{\varphi} \\
G \arrow{r}{t} & \Gamma \\
\end{tikzcd}
$$
commutative. The pair $(t,s)$ then gives rise to a map of pairs $B(t,s)\colon (BG,BH) \rightarrow (B\Gamma,B\Lambda)$ and the map
$$\mathrm{B}(t,s)_*\colon \mathrm{Pos}^{\mathrm{Spin}}_n(\mathrm{B}f) \rightarrow \mathrm{Pos}^{\mathrm{Spin}}_n(\mathrm{B}\varphi)$$
is obtained by postcomposing the classifying maps for a cycle of $\mathrm{Pos}^{\mathrm{Spin}}_n(\mathrm{B}f)$ with $\mathrm{B}(t,s)$. Moreover, the commutativity of the diagram implies that $t$ induces a homomorphism $p\colon  \frac{G}{\mathrm{Im}(H)^G} \rightarrow \frac{\Gamma}{\mathrm{Im}(\Lambda)^\Gamma}$.
\begin{proposition}
Let $G$ and $H$ be as above and suppose that the homomorphism $p\colon \frac{G}{\mathrm{Im}(H)^G} \rightarrow \frac{\Gamma}{\mathrm{Im}(\Lambda)^\Gamma}$ is injective.
Then the diagram
$$
\begin{tikzcd}
\mathrm{Pos}^{\mathrm{Spin}}_n(\mathrm{B}f) \arrow{rr} {\mathrm{B}(t,s)_*}\arrow{rd}[swap]{\rho^{f}_{(2)}} & & \mathrm{Pos}^{\mathrm{Spin}}_n(\mathrm{B}\varphi) \arrow{ld}{\rho^{\varphi}_{(2)}}\\
& \reals & \\
\end{tikzcd}
$$
is commutative.
\label{prop:injectiverhotwocommu}
\end{proposition}
\begin{proof}
Denote by $\pi_{r}^{G}$ and $\pi_{r}^{\Gamma}$ the canonical projection maps $G \rightarrow G^{}/\mathrm{Im}(H)^{G}$ and $\Gamma \rightarrow \Gamma^{}/\mathrm{Im}(\Lambda)^{\Gamma}$, respectively. Let $\pi_{\mathrm{triv}}^{G}\colon  G \rightarrow G^{}/\mathrm{Im}(H)^{G}$ and $\pi_{\mathrm{triv}}^{\Gamma}\colon \Gamma \rightarrow \Gamma^{}/\mathrm{Im}(\Lambda)^{\Gamma}$ be the trivial group homomorphisms.
Consider the diagram
$$
\begin{tikzcd}
& \mathrm{B}(G*_HG)\arrow{rr}{\mathrm{B}(\pi_{r}^{G}*\pi_{\mathrm{triv}}^{G})}\arrow{dd} & & \mathrm{B}(\frac{G}{\mathrm{Im}(H)^G}) \arrow{dd}{\mathrm{B}p} \\
\mathrm{D}M\arrow{ru}{\mathrm{D}c} \arrow{rd}[swap]{\mathrm{D}(B(t,s))\circ c} & & \\
& \mathrm{B}(\Gamma*_\Lambda\Gamma) \arrow{rr}{\mathrm{B}(\pi_{r}^{\Gamma}*\pi_{\mathrm{triv}}^{\Gamma})} & & \mathrm{B}\left(\frac{\Gamma}{\mathrm{Im}(\varphi)^\Gamma}\right)\\
\end{tikzcd},
$$
where the maps $\mathrm{D}c$ and $\mathrm{D}(B(t,s)\circ c)$ denote the ``doubles" of $c$ and $\mathrm{B}(t,s)\circ c$ and where the maps of the classifying spaces are induced by the maps of the corresponding groups. The induced maps of the classifying spaces can be chosen in such a way to make the diagram commutative up to homotopy.

Let $[(M,N,g,c)]$ be an element of $\mathrm{Pos}^{\mathrm{Spin}}_n(\mathrm{B}f)$. We need to show that $\rho^{f}_{(2)}(M,N,g,c) = \rho^{\varphi}_{(2)}(M,N,g,\mathrm{B}(t,s)\circ c)$. From Proposition \ref{prop:relrhol2eta}, we know that $\rho^{f}_{(2)}(M,N,g,c)$ is the $L^2$-$\eta$-invariant of the cover of $\mathrm{D}M$ associated to $\mathrm{B}(\pi_{r}^{G}*\pi_{\mathrm{triv}}^{G}) \circ \mathrm{D}c$ and $\rho_2(M,N,g,\mathrm{B}(t,s)\circ c)$ is the $L^2$ $\eta$-invariant of the cover of $\mathrm{D}M$ associated to $\mathrm{B}(\pi_{r}^{\Gamma}*\pi_{\mathrm{triv}}^{\Gamma}) \circ \mathrm{D}(\mathrm{B}(t,s)\circ c)$. The (homotopy) commutativity of the above diagram implies that the cover of $\mathrm{D}M$ associated to $\mathrm{B}(\pi_{r}^{\Gamma}*\pi_{\mathrm{triv}}^{\Gamma}) \circ \mathrm{D}(\mathrm{B}(t,s)\circ c)$ is the same as the cover associated to $\mathrm{B}p \circ \mathrm{B}(\pi_{r}^{G}*\pi_{\mathrm{triv}}^{G}) \circ \mathrm{D}c$. Since $p$ is injective, the latter cover is a disjoint union of copies of the cover of $\mathrm{D}M$ associated to the map $\mathrm{B}(\pi_{r}^{G}*\pi_{\mathrm{triv}}^{G}) \circ \mathrm{D}c$. The claim then follows from the explicit formula  \eqref{eta:formula} for the $L^2$-$\eta$-invariant.
\end{proof}
In the rest of this section, $M$ will denote a compact odd dimensional spin manifold with boundary $N$. Denote by $\mathcal{R}^{\mathrm{sc}>0}(M,N)$ the space of positive scalar curvature metrics on $M$ which are collared near the boundary. Denote by $\iota_*$ the map $\pi_1(N) \rightarrow \pi_1(M)$ induced by the inclusion $\iota\colon  N \rightarrow M$. Let $c\colon  M \rightarrow \mathrm{B}\Gamma$ and $c_N\colon N \rightarrow \mathrm{B}\Lambda$ denote the classifying maps of the universal covers of $M$ and $N$, respectively. Consider $\mathrm{B}\iota_*\colon  \mathrm{B}\Lambda \rightarrow \mathrm{B}\Gamma$, which we will assume to be injective. The latter maps can be chosen to make the diagram
$$
\begin{tikzcd}
N \arrow{r}{c_N} \arrow{d}{\iota} & \mathrm{B}\Lambda \arrow{d}{\mathrm{B}\iota_*}\\
M \arrow{r}{c} & \mathrm{B}\Gamma
\end{tikzcd}.
$$
Thus, $c$ can be seen as a map of pairs $(M,N)\rightarrow (\mathrm{B}\Gamma,\mathrm{B}\Lambda)$.
There is a canonical map
$$\mathrm{cl}\colon \mathcal{R}^{\mathrm{sc}>0}(M,N) \rightarrow \mathrm{Pos}^{\mathrm{Spin}}_{\mathrm{dim} M}(\mathrm{B}\iota_*)$$
sending a psc metric $g$ to the cycle represented by $[(M,N,c,g)]$.
\begin{definition}
Two metrics $g_1,g_2 \in \mathcal{R}^{\mathrm{sc}>0}(M,N)$ are said to be $(\pi_1(M),\pi_1(N))$-bordant if $\mathrm{cl}(g_1) = \mathrm{cl}(g_2)$. Moreover, $(\pi_1(M),\pi_1(N))$-bordism defines an equivalence relation on $\mathcal{R}^{\mathrm{sc}>0}(M,N)$. We will denote the quotient by $\mathcal{R}_{\mathrm{bord}}^{\mathrm{sc}>0}(M,N)$.
\end{definition}
\begin{remark}
Note that $\mathrm{cl}\colon \mathcal{R}^{\mathrm{sc}>0}(M,N) \rightarrow \mathrm{Pos}^{\mathrm{Spin}}_{\mathrm{dim} M}(\mathrm{B}\iota_*)$ induces an inclusion of $\mathcal{R}_{\mathrm{bord}}^{\mathrm{sc}>0}(M,N)$ into $\mathrm{Pos}^{\mathrm{Spin}}_{\mathrm{dim} M}(\mathrm{B}\iota_*)$.
\end{remark}
Denote by $\mathrm{Diff}^{\mathrm{Spin}}(M)$ the group of spin structure preserving diffeomorphisms of the manifold $M$. This group acts on $\mathcal{R}^{\mathrm{sc}>0}(M,N)$ via pullbacks. Proposition \ref{prop:constantonorbits} is an analogue of \cite{PS07II}*{Proposition 2.10} and in fact its proof can be almost reduced to the proof of \cite{PS07II}*{Proposition 2.10} by Piazza and Schick.
\begin{proposition}
The composition
$$\mathcal{R}^{\mathrm{sc}>0}(M,N) \rightarrow \mathcal{R}_{\mathrm{bord}}^{\mathrm{sc}>0}(M,N) \hookrightarrow \mathrm{Pos}^{\mathrm{Spin}}_{\mathrm{dim} M}(\mathrm{B}\iota_*) \xrightarrow{\rho^{\iota_*}_{(2)}} \mathbb{R}$$
is constant on the orbits of the action of $\mathrm{Diff}^{\mathrm{Spin}}(M)$.
\label{prop:constantonorbits}
\end{proposition}
\begin{proof}
We only sketch the proof. Let $\Psi \in \mathrm{Diff}^{\mathrm{Spin}}(M)$. By ``doubling" $\Psi$, we obtain a spin structure preserving diffeomorphism of the double $\mathrm{D}M$ of $M$. Let $g \in \mathcal{R}^{\mathrm{sc}>0}(M,N)$ and denote by $\mathrm{D}g$ the metric on $\mathrm{D}M$ obtained by doubling $g$. 

By \cite{PS07II}*{Proposition 2.10} the spin Dirac operators on the universal cover $\widetilde{\mathrm{D}M}$ of $\mathrm{D}M$ constructed using the pullbacks of $\mathrm{D}g\in \mathcal{R}^{\mathrm{sc}>0}(M,N)$ and $(\mathrm{D}\Psi)^*(\mathrm{D}g)$ are unitarily equivalent.
 We thus have a unitary equivalence of the spin Dirac operators on the quotient of $\widetilde{\mathrm{D}M}$ by the kernel of the homomorphism $\Gamma *_\Lambda \Gamma \rightarrow \Gamma$ defined in the beginning of Section \ref{sec:defrelrho}.
 
  Hence, the $L^2$-$\eta$-invariant of the latter operators are equal. On the other hand, by Proposition \ref{prop:relrhol2eta}, these $L^2$-$\eta$-invariants coincide with the relative $L^2$-$\rho$-invariants associated to the psc metrics $g$ and $\Psi^*g$. The claim follows.
\end{proof}
\begin{corollary}
The map $\mathcal{R}^{\mathrm{sc}>0}(M,N)\to \mathbb{R}$ of Proposition \ref{prop:constantonorbits} descends to a map $\pi_0(\mathcal{R}^{\mathrm{sc}>0}(M,N)^{}/\mathrm{Diff}^{\mathrm{Spin}}(M)) \rightarrow \mathbb{R}$.
\end{corollary}
\section{Geometric Application}
We now discuss how our relative $L^2$-$\rho$-invariant can be used to make statements about the richness of the (moduli) space of positive scalar curvature metrics on manifolds with boundary which are collared near the boundary.
\begin{lemma}
Let $f\colon 1\rightarrow \mathbb{Z}_m$ be the inclusion of the trivial group into $\mathbb{Z}_m$ for $m \in \mathbb{N}$. Let $n = 4k +3 > 4$. Denote by $K^f$ the kernel of the map $\mathrm{Pos}^{\mathrm{Spin}}_n(\mathrm{B}f) \rightarrow \Omega^{\mathrm{Spin}}_n(\mathrm{B}f)$. Then $\rho^{f}_{(2)}$ restricted to $K^f$ is nontrivial. In particular, $K^f$ is infinite.
\label{lem:Kfinfinite}
\end{lemma}
\begin{proof}
Denote by $K$ the kernel of the map $\mathrm{Pos}^{\mathrm{Spin}}_n(\mathrm{B}\mathbb{Z}_m) \rightarrow \Omega^{\mathrm{Spin}}_n(\mathrm{B}\mathbb{Z}_m)$. From \cite{BG95}*{Lemma 2.3} and \cite{PS07II}*{Example 2.13}
it is known that the map $\rho^{\mathbb{Z}_m}_{(2)}$ is nontrivial. The fact that $\Omega^{\mathrm{Spin}}_n(\mathbb{Z}_m)$ is finite, then implies that $\rho^{\mathbb{Z}_m}_{(2)}|_K$ is nontrivial (see the beginning of the proof of \cite{PS07II}*{Theorem 2.25}). The commutativity of the diagram
$$
\begin{tikzcd}
\mathrm{Pos}^{\mathrm{Spin}}_n(\mathrm{B}\mathbb{Z}_m) \arrow{r} \arrow{d} & \Omega^{\mathrm{Spin}}_n(\mathrm{B}\mathbb{Z}_m) \arrow{d} \\
\mathrm{Pos}^{\mathrm{Spin}}_n(\mathrm{B}f) \arrow{r} & \Omega^{\mathrm{Spin}}_n(\mathrm{B}f)
\end{tikzcd}
$$
implies that the image of $K$ under the map $\mathrm{Pos}^{\mathrm{Spin}}_n(\mathrm{B}\mathbb{Z}_m) \rightarrow \mathrm{Pos}^{\mathrm{Spin}}_n(\mathrm{B}f)$ lies in $K^f$. The claim follows from this observation and Proposition \ref{prop:Gfcommu}.
\end{proof}
The following proposition follows immediately from the proof of Lemma \ref{lem:Kfinfinite}.
\begin{proposition}
We use the notation of Lemma \ref{lem:Kfinfinite}. There exists an infinite subset of $K^{f}$ with the following properties:
\begin{itemize}
\item its elements have pairwise different relative $L^2$-$\rho$-invariants and
\item are represented by a cycle $(M,N,c,c|_N,g)$ with $\partial M = N = \varnothing$.
\end{itemize}
\label{prop:closedrepresentatives}
\end{proposition}
\begin{definition}
Let $\varphi\colon \Lambda \rightarrow \Gamma$ be a group homomorphism. An element $\gamma \in \Gamma$ is called a relative torsion element of order $m$ of the homomorphism $\varphi$ if it is a torsion element of order $m$ and if the subgroup generated by $\gamma$ intersects $\mathrm{Im}(\Lambda)^\Gamma$ only at the neutral element. If such an element exists, we say that $\varphi$ has relative torsion.
\end{definition}
\begin{proposition}
Let $\varphi\colon \Lambda \rightarrow \Gamma$ be a group homomorphism. Suppose that $\varphi$ has relative torsion. Let $n = 4k +3 > 4$. Denote by $K^{\varphi}$ the kernel of the map $\mathrm{Pos}^{\mathrm{Spin}}_n(\mathrm{B}\varphi) \rightarrow \Omega^{\mathrm{Spin}}_n(\mathrm{B}\varphi)$. Then there exists an infinite subset of $K^{\varphi}$ with the following properties:
\begin{itemize}
\item its elements have pairwise different relative $L^2$-$\rho$-invariants and
\item are represented by a cycle $(M,N,c,c|_N,g)$ with $\partial M = N = \varnothing$.
\end{itemize}
\label{prop:closedrepresentativesgammalambda}
\end{proposition}
\begin{proof}
Let $\gamma \in \Gamma$ be a relative torsion element of order $m$ of $\varphi$. Sending the generator of $\mathbb{Z}_m$ to $\gamma$ gives an injective homomorphism $t\colon \mathbb{Z}_m \rightarrow \Gamma$. Denote by $s\colon 1 \rightarrow \Lambda$ the trivial homomorphism of the trivial group into $\Lambda$. The pair $(t,s)$ gives rise to map $\mathrm{B}(t,s)_*\colon  \mathrm{Pos}^{\mathrm{Spin}}_{n}(\mathrm{B}f) \rightarrow \mathrm{Pos}^{\mathrm{Spin}}_{n}(\mathrm{B}\varphi)$, where $f\colon 1 \rightarrow \mathbb{Z}_m$ is the trivial homomorphism. Postcomposition of the map $t$ with the canonical projection $\Gamma \rightarrow \Gamma^{}/\mathrm{Im}(\Lambda)^\Gamma$ gives rise to an injective map. Thus, by Proposition~\ref{prop:injectiverhotwocommu}, the diagram
$$
\begin{tikzcd}
\mathrm{Pos}^{\mathrm{Spin}}_n(\mathrm{B}f) \arrow{rr} {\mathrm{B}(t,s)_*}\arrow{rd}[swap]{\rho^{f}_{(2)}} & & \mathrm{Pos}^{\mathrm{Spin}}_n(\mathrm{B}\varphi) \arrow{ld}{\rho^{\varphi}_{(2)}}\\
& \reals & \\
\end{tikzcd}
$$
is commutative and the subset of $K^\varphi$ of the claim can be obtained by pushing the subset of $K^f$ provided by Proposition \ref{prop:closedrepresentatives} to $\mathrm{Pos}^{\mathrm{Spin}}_n(\mathrm{B}\varphi)$ using $\mathrm{B}(t,s)_*$. The fact that these elements lie in $K^{\varphi}$ follows from the naturality of the relative Stolz positive scalar curvature sequence.
\end{proof}
As in the previous section, let $M$ be a compact spin odd-dimensional manifold with boundary $N$ and let $\mathrm{Diff}^{\mathrm{Spin}}(M)$ denote the group of spin structure preserving diffeomorphisms of $M$. 
Furthermore, denote by $\iota_*$ the map $\pi_1(N) \rightarrow \pi_1(M)$ induced by the inclusion and by $c$ the map of pairs $(M,N) \rightarrow (\mathrm{B}\pi_1(M),\mathrm{B}\pi_1(N))$ classifying the universal covers of $M$ and $N$.
\begin{theorem}
Suppose that $n\coloneqq\mathrm{dim} M > 4$ is of the form $4k+3$ and that $\iota_*$ has relative torsion. Furthermore, suppose that $\mathcal{R}^{\mathrm{sc}>0}(M,N) \neq \varnothing$. Then there are infinitely many orbits of the action of $\mathrm{Diff}^{\mathrm{Spin}}(M)$ on $\mathcal{R}^{\mathrm{sc}>0}(M,N)$ such that metrics belonging to different orbits are not $(\pi_1(M),\pi_1(N))$-bordant.
\label{thm:geometricresult}
\end{theorem}
\begin{proof}
The statement of the theorem follows from Proposition \ref{prop:constantonorbits} and the following claim.
\begin{claima}
There are infinitely many classes in $\mathcal{R}_{\mathrm{bord}}^{\mathrm{sc}>0}(M,N)$ with different relative $L^2$-$\rho$-invariants.
\end{claima}

In the following, we will prove Claim A. Denote by $K^{\iota_*}$ the kernel of 
$$\mathrm{Pos}^{\mathrm{Spin}}_n(\mathrm{B}\iota_*) \rightarrow \mathrm{\Omega}^{\mathrm{Spin}}_n(\mathrm{B}\iota_*)$$ and let $S$ be an infinite subset of $K^{\iota_*}$ provided by Proposition \ref{prop:closedrepresentativesgammalambda}. Furthermore, let $g_0 \in \mathcal{R}^{\mathrm{sc}>0}(M,N)$ be some psc metric on $M$. For any $s = [(M^\prime,\varnothing,c^\prime,\varnothing\rightarrow \mathrm{B}{\pi_1(N)},g^\prime)] \in S$ we have
$$\rho^{\iota_*}_{(2)}(s + [(M,N,c,g_0)]) = \rho^{\iota_*}_{(2)}(s) + \rho^{\iota_*}_{(2)}( [(M,N,c,g_0)]).$$
Thus $s + [(M,N,c,g_0)] \neq s^\prime +[(M,N,c,g_0)]$ if $s \neq s^\prime$. Claim A then follows from
\begin{claimb}
For each $s = [(M^\prime,\varnothing,c^\prime,\varnothing\rightarrow \mathrm{B}{\pi_1(N)},g^\prime)] \in S$ there exists $g_s \in \mathcal{R}^{\mathrm{sc}>0}(M,N)$ such that $s + [(M,N,c,g_0)] = [(M,N,c,g_s)] \in \mathrm{Pos}^{\mathrm{Spin}}_n(\mathrm{B}\iota_*)$.
\end{claimb}

In the following, we will prove Claim B. Claim B is a relative version of \cite{BG95}*{Lemma 3.1} (see also \cite{PS07II}*{Proposition 2.4}). Botvinnik and Gilkey cite the results of \cite{Mi84} in the proof of \cite{BG95}*{Lemma 3.1}. Since $s$ is represented by a closed manifold, we can use essentially the same arguments. We sketch the arguments for the convenience of the reader and refer the reader to \cite{Mi84} for further details.

Note that since $s \in K^{\iota_*}$ there exists $(V^\prime,E^\prime)$ with $\partial V^\prime = M^\prime$ and $E^\prime\colon V^\prime \rightarrow \mathrm{B}\Gamma$ a continuous map extending $c^\prime$. Thus $[(M^\prime,\varnothing,c^\prime,\varnothing\rightarrow \mathrm{B}{\pi_1(N)})] + [(M,N,c)] = [(M,N,c)] \in \mathrm{\Omega}^{\mathrm{Spin}}_n(\mathrm{B}\iota_*)$. An explicit equivalence (bordism) between the cycles $(M^\prime,\varnothing,c^\prime,\varnothing\rightarrow \mathrm{B}{\pi_1(N)}) + (M,N,c)$ and $(M,N,c)$ is given by $(V \coloneqq V^\prime \sqcup M \times [0,1],E)$, where $E\colon  (V,N \times [0,1]) \rightarrow (\mathrm{B}\pi_1(M),\mathrm{B}\pi_1(N))$ is the map of pairs which restricts to $E^\prime$ on $V^\prime$ and to $c \circ \pi_M$ on $M \times [0,1]$. Here, $\pi_M\colon M \times [0,1] \rightarrow M$ denotes the projection on $M$. Consider the connected sum of $V^\prime$ and $M\times [0,1]$ obtained by removing a disk from $V^\prime$ away from $M^\prime$ and a disk on $M \times [0,1]$ away from $\partial (M \times [0,1])$ and then attaching the rest along the boundary of the removed disks. We denote the resulting space again by $V$. Note that, since any two maps on a disk are homotopic we can define a map, also denoted by $E$, out of the new $V$ to $\mathrm{B}\pi_1(M)$ which coincides with the previous map on the complement of the removed disks. To summarise, we have a connected spin bordism $V$ between $M\times\{0\} \sqcup M^\prime$ and $M\times\{1\}$ with internal boundary $N \times [0,1]$ and a map $E\colon (V,N\times[0,1]) \rightarrow (\mathrm{B}\pi_1(M),\mathrm{B}\pi_1(N))$.
Thanks to \cite[Lemma 5.2]{Stolz} (see also \cite[Proposition 3.2]{Schick-Zenobi}) and the fact that the surgeries involved in its proof can be done away from $\partial V$, we can assume that 
	\begin{itemize}
		\item $E\colon V\rightarrow \mathrm{B}\pi_1(M)$ induces an isomorphism at the level of fundamental groups;
		\item the inclusion $M\times\{1\} \hookrightarrow V$ is $2$-connected (i.e. it induces an epimorphism between $\pi_2$-groups).
	\end{itemize}
 From \cite{Wa71}\footnote{bordisms of manifolds with boundary such as $V$ are dealt with in \cite{Wa71}. There, they are called cobordisms relative to the boundary.} (see also \cite{Mi84}*{Theorem 1.1}), it follows that there exists a self-indexing Morse function on the bordism $V$ (see \cite{Mi84} for the definition) which takes the value $-1/2$ on $M\times \{1\}$ and the value $n+1+1/2$ on $M^\prime \sqcup M \times\{0\}$, has no critical points of index less than $3$ and has no critical points near $\partial V$. Furthermore, we can assume that in a collar neighbourhood of $N\times [0,1]$ it is a linear interpolation between $-1/2$ and $n+1+1/2$. Using $-f$, we can decompose $V$ into traces of surgeries on embedded spheres of codimension bigger than $2$. Note that, since $-f$ has no critical points near $N \times [0,1]$ the result of each surgery has boundary $N$. The improved surgery theorem (see \cite{Ga87} and \cite{Wal11}*{Theorem 0.2}) then implies that $V$ can be endowed with a psc metric $G$, which is collared near $M\times\{0\} \sqcup M^\prime$ and $M \times \{1\}$ and extends the metric $g^\prime \sqcup g_0$. Moreover, since the surgeries on the level sets are taking place  away from $N$, following the construction in \cite{Wal11} gives rise to $G$ which is collared near $N \times [0,1]$ and even restricts to $g_0|_N + \mathrm{d}t^2$ on $N\times [0,1]$. Setting $g_s \coloneqq G|_{M\times \{1\}}$ then finishes the proof of Claim A and the theorem.
\end{proof}
Theorem \ref{thm:geometricresult} has the following immediate
\begin{corollary}
Suppose that $M$ satisfies the conditions of Theorem \ref{thm:geometricresult} then
\begin{itemize}
\item $\#\mathcal{R}_{\mathrm{bord}}^{\mathrm{sc}>0}(M,N) = \infty$
\item $\#\pi_0(\mathcal{R}^{\mathrm{sc}>0}(M,N)) = \infty$
\item $\#\pi_0(\mathcal{R}^{\mathrm{sc}>0}(M,N)^{}/\mathrm{Diff}^{\mathrm{Spin}}(M)) = \infty$
\end{itemize}
\end{corollary}
\begin{remark}
Suppose that $M$ satisfies the conditions of Theorem \ref{thm:geometricresult} and $g \in \mathcal{R}^{\mathrm{sc}>0}(M,N)$. Set $h \coloneqq g|_N$ and denote by $\mathcal{R}^{\mathrm{sc}>0}(M,N)_h$ the subspace of $\mathcal{R}^{\mathrm{sc}>0}(M,N)$ consisting of metrics which restrict to $h$ on the boundary. Then it follows from the proof of Theorem \ref{thm:geometricresult} that $$\mathcal{R}^{\mathrm{sc}>0}(M,N)_h$$ also has infinitely many connected components. The subgroup of $\mathrm{Diff}^{\mathrm{Spin}}(M)$ consisting of diffeomorphisms which restrict to the identity map of $N$ acts on $\mathcal{R}^{\mathrm{sc}>0}(M,N)_h$. Same observations imply that the quotient of $\mathcal{R}^{\mathrm{sc}>0}(M,N)_h$ by this subgroup also has infinitely many connected components.
\end{remark}
\begin{remark}
In \cite{HR10} and \cite{BR15} the higher $\rho$-invariants associated to psc metrics on closed manifolds have been related to the classical  APS and $L^2$-$\rho$-invariants and this relationship has been used to prove certain rigidity results for these numerical invariants. Recently, $K$-theoretic higher $\rho$-invariants associated to psc metrics on manifolds with boundary have been defined in \cite{PZ19} and \cite{S20} using different techniques. We plan to address the relationship between the latter $K$-theoretic invariants and the invariants defined in this paper in the future.
\end{remark}
\bibliography{references}

\begin{bibdiv}
\begin{biblist}

\bib{APS75I}{article}{
      author={Atiyah, M.~F.},
      author={Patodi, V.~K.},
      author={Singer, I.~M.},
       title={Spectral asymmetry and {R}iemannian geometry. {I}},
        date={1975},
        ISSN={0305-0041},
     journal={Math. Proc. Cambridge Philos. Soc.},
      volume={77},
       pages={43\ndash 69},
         url={https://doi.org/10.1017/S0305004100049410},
      review={\MR{397797}},
}

\bib{APSII75}{article}{
      author={Atiyah, M.~F.},
      author={Patodi, V.~K.},
      author={Singer, I.~M.},
       title={Spectral asymmetry and {R}iemannian geometry. {II}},
        date={1975},
        ISSN={0305-0041},
     journal={Math. Proc. Cambridge Philos. Soc.},
      volume={78},
      number={3},
       pages={405\ndash 432},
         url={https://doi.org/10.1017/S0305004100051872},
      review={\MR{397798}},
}

\bib{BR15}{article}{
      author={Benameur, Moulay-Tahar},
      author={Roy, Indrava},
       title={The {H}igson-{R}oe exact sequence and {$\ell^2$} eta invariants},
        date={2015},
        ISSN={0022-1236},
     journal={J. Funct. Anal.},
      volume={268},
      number={4},
       pages={974\ndash 1031},
         url={https://doi.org/10.1016/j.jfa.2014.11.006},
      review={\MR{3296587}},
}

\bib{BERW17}{article}{
      author={Botvinnik, Boris},
      author={Ebert, Johannes},
      author={Randal-Williams, Oscar},
       title={Infinite loop spaces and positive scalar curvature},
        date={2017},
        ISSN={0020-9910},
     journal={Invent. Math.},
      volume={209},
      number={3},
       pages={749\ndash 835},
         url={https://doi.org/10.1007/s00222-017-0719-3},
      review={\MR{3681394}},
}

\bib{BG95}{article}{
      author={Botvinnik, Boris},
      author={Gilkey, Peter~B.},
       title={The eta invariant and metrics of positive scalar curvature},
        date={1995},
        ISSN={0025-5831},
     journal={Math. Ann.},
      volume={302},
      number={3},
       pages={507\ndash 517},
         url={https://doi.org/10.1007/BF01444505},
      review={\MR{1339924}},
}

\bib{B15}{article}{
      author={Bunke, Ulrich},
       title={The universal {$\eta$}-invariant for manifolds with boundary},
        date={2015},
        ISSN={0033-5606},
     journal={Q. J. Math.},
      volume={66},
      number={2},
       pages={473\ndash 506},
         url={https://doi.org/10.1093/qmath/hav002},
      review={\MR{3356833}},
}

\bib{CG85}{article}{
      author={Cheeger, Jeff},
      author={Gromov, Mikhael},
       title={Bounds on the von {N}eumann dimension of $l\sp 2$-cohomology and
  the gauss-bonnet theorem for open manifolds},
        date={1985},
     journal={J. Differential Geom.},
      volume={21},
      number={1},
       pages={1\ndash 34},
         url={https://doi.org/10.4310/jdg/1214439461},
}

\bib{ERW19}{article}{
      author={Ebert, Johannes},
      author={Randal-Williams, Oscar},
       title={Infinite loop spaces and positive scalar curvature in the
  presence of a fundamental group},
        date={2019},
        ISSN={1465-3060},
     journal={Geom. Topol.},
      volume={23},
      number={3},
       pages={1549\ndash 1610},
         url={https://doi.org/10.2140/gt.2019.23.1549},
      review={\MR{3956897}},
}

\bib{Ga87}{article}{
      author={Gajer, Pawe\l},
       title={Riemannian metrics of positive scalar curvature on compact
  manifolds with boundary},
        date={1987},
        ISSN={0232-704X},
     journal={Ann. Global Anal. Geom.},
      volume={5},
      number={3},
       pages={179\ndash 191},
         url={https://doi.org/10.1007/BF00128019},
      review={\MR{962295}},
}

\bib{HR10}{article}{
      author={Higson, Nigel},
      author={Roe, John},
       title={{$K$}-homology, assembly and rigidity theorems for relative eta
  invariants},
        date={2010},
        ISSN={1558-8599},
     journal={Pure Appl. Math. Q.},
      volume={6},
      number={2, Special Issue: In honor of Michael Atiyah and Isadore Singer},
       pages={555\ndash 601},
         url={https://doi.org/10.4310/PAMQ.2010.v6.n2.a11},
      review={\MR{2761858}},
}

\bib{LM89}{book}{
      author={Lawson, H.~Blaine, Jr.},
      author={Michelsohn, Marie-Louise},
       title={Spin geometry},
      series={Princeton Mathematical Series},
   publisher={Princeton University Press, Princeton, NJ},
        date={1989},
      volume={38},
        ISBN={0-691-08542-0},
      review={\MR{1031992}},
}

\bib{L92}{article}{
      author={Lott, J.},
       title={Superconnections and higher index theory},
        date={1992},
        ISSN={1016-443X},
     journal={Geom. Funct. Anal.},
      volume={2},
      number={4},
       pages={421\ndash 454},
         url={https://doi.org/10.1007/BF01896662},
      review={\MR{1191568}},
}

\bib{L92II}{article}{
      author={Lott, John},
       title={Higher eta-invariants},
        date={1992},
        ISSN={0920-3036},
     journal={$K$-Theory},
      volume={6},
      number={3},
       pages={191\ndash 233},
         url={https://doi.org/10.1007/BF00961464},
      review={\MR{1189276}},
}

\bib{Mi84}{article}{
      author={Miyazaki, Tetsuro},
       title={On the existence of positive scalar curvature metrics on
  non-simply-connected manifolds},
        date={1984},
        ISSN={0040-8980},
     journal={J. Fac. Sci. Univ. Tokyo Sect. IA Math.},
      volume={30},
      number={3},
       pages={549\ndash 561},
      review={\MR{731517}},
}

\bib{PS07}{article}{
      author={Piazza, Paolo},
      author={Schick, Thomas},
       title={Bordism, rho-invariants and the {B}aum-{C}onnes conjecture},
        date={2007},
        ISSN={1661-6952},
     journal={J. Noncommut. Geom.},
      volume={1},
      number={1},
       pages={27\ndash 111},
         url={https://doi.org/10.4171/JNCG/2},
      review={\MR{2294190}},
}

\bib{PS07II}{article}{
      author={Piazza, Paolo},
      author={Schick, Thomas},
       title={Groups with torsion, bordism and rho invariants},
        date={2007},
        ISSN={0030-8730},
     journal={Pacific J. Math.},
      volume={232},
      number={2},
       pages={355\ndash 378},
         url={https://doi.org/10.2140/pjm.2007.232.355},
      review={\MR{2366359}},
}

\bib{PS14}{article}{
      author={Piazza, Paolo},
      author={Schick, Thomas},
       title={Rho-classes, index theory and {S}tolz' positive scalar curvature
  sequence},
        date={2014},
        ISSN={1753-8416},
     journal={J. Topol.},
      volume={7},
      number={4},
       pages={965\ndash 1004},
         url={https://doi.org/10.1112/jtopol/jtt048},
      review={\MR{3286895}},
}

\bib{PZ19}{article}{
      author={Piazza, Paolo},
      author={Zenobi, Vito~Felice},
       title={Singular spaces, groupoids and metrics of positive scalar
  curvature},
        date={2019},
        ISSN={0393-0440},
     journal={J. Geom. Phys.},
      volume={137},
       pages={87\ndash 123},
         url={https://doi.org/10.1016/j.geomphys.2018.09.016},
      review={\MR{3893404}},
}

\bib{RS01}{incollection}{
      author={Rosenberg, Jonathan},
      author={Stolz, Stephan},
       title={Metrics of positive scalar curvature and connections with
  surgery},
        date={2001},
   booktitle={Surveys on surgery theory, {V}ol. 2},
      series={Ann. of Math. Stud.},
      volume={149},
   publisher={Princeton Univ. Press, Princeton, NJ},
       pages={353\ndash 386},
      review={\MR{1818778}},
}

\bib{Schick-Zenobi}{unpublished}{
      author={Schick, Thoma},
      author={Zenobi, Vito~Felice},
       title={Positive scalar curvature due to the cokernel of the classifying
  map},
        note={arXiv:2006.15965},
}

\bib{Sch05}{article}{
      author={Schick, Thomas},
       title={{$L^2$}-index theorems, {$KK$}-theory, and connections},
        date={2005},
     journal={New York J. Math.},
      volume={11},
       pages={387\ndash 443},
         url={http://nyjm.albany.edu:8000/j/2005/11_387.html},
      review={\MR{2188248}},
}

\bib{S20}{article}{
      author={{Seyedhosseini}, Mehran},
       title={{A Variant of Roe Algebras for Spaces with Cylindrical Ends with
  Applications in Relative Higher Index Theory}},
        date={2020-03},
     journal={arXiv e-prints},
       pages={arXiv:2003.07993},
      eprint={2003.07993},
}

\bib{Stolz}{unpublished}{
      author={Stolz, Stephan},
       title={Concordance classes of positive scalar curvature metrics},
        note={\href{ url:
  http://www3.nd.edu/~stolz/concordance.ps}{http://www3.nd.edu/~stolz/concordance.ps}},
}

\bib{Wa71}{article}{
      author={Wall, C. T.~C.},
       title={Geometrical connectivity. {I}},
        date={1971},
        ISSN={0024-6107},
     journal={J. London Math. Soc. (2)},
      volume={3},
       pages={597\ndash 604},
         url={https://doi.org/10.1112/jlms/s2-3.4.597},
      review={\MR{290387}},
}

\bib{Wal11}{article}{
      author={Walsh, Mark},
       title={Metrics of positive scalar curvature and generalised {M}orse
  functions, {P}art {I}},
        date={2011},
        ISSN={0065-9266},
     journal={Mem. Amer. Math. Soc.},
      volume={209},
      number={983},
       pages={xviii+80},
         url={https://doi.org/10.1090/S0065-9266-10-00622-8},
      review={\MR{2789750}},
}

\bib{XYZ17}{article}{
      author={{Xie}, Zhizhang},
      author={{Yu}, Guoliang},
      author={{Zeidler}, Rudolf},
       title={{On the range of the relative higher index and the higher
  rho-invariant for positive scalar curvature}},
        date={2017-12},
     journal={arXiv e-prints},
       pages={arXiv:1712.03722},
      eprint={1712.03722},
}

\end{biblist}
\end{bibdiv}
\end{document}